\documentclass[11pt, a4paper]{amsart}
\usepackage{amscd, amsfonts, amsmath, amsrefs, amsthm, amssymb, booktabs, cases, color, comment, dsfont, graphics, graphicx, hyperref, latexsym, longtable, mathptmx, mathtools, microtype, tikz, tikz-cd, times} 

\newcommand{\bb}{\mathbb}

\newcommand{\Z}{\mathbb Z}
\newcommand{\Q}{\mathbb Q}

\newcommand{\C}{\mathbb C}
\newcommand{\F}{\mathbb F}
\newcommand{\G}{\mathbb G}
\newcommand{\cL}{\mathcal L}
\newcommand{\Fp}{\mathbb F_p}
\newcommand{\Fq}{\mathbb F_q}

\newcommand{\kln}{\textrm{Kl}_n}
\newcommand{\smlm}{\sum\limits_}

\newcommand{\Qlb}{\overline{\mathbb Q}_{\ell}}
\theoremstyle{plain}
\newtheorem{theorem}{Theorem}

\newtheorem{proposition}[theorem]{Proposition}

\theoremstyle{definition}

\newtheorem{remark}[theorem]{Remark}

\theoremstyle{definition}

\newcommand{\abs}[1]{\lvert#1\rvert}
\newtheorem{defn}[theorem]{Definition}

\begin{document}
\title[On equidistribution of Gauss sums of cuspidal representations of 
$GL_d(\F_q)$]{On equidistribution of Gauss sums of \\cuspidal representations of 
$GL_d(\F_q)$}

\author{Sameer Kulkarni}

\address{Tata Institute of Fundamental  Research, Homi Bhabha Road,
	Bombay - 400 005, INDIA.}  
\email{sameer@math.tifr.res.in}

\author{C.~S.~Rajan}

\address{Tata Institute of Fundamental  Research, Homi Bhabha Road,
	Bombay - 400 005, INDIA.}  \email{rajan@math.tifr.res.in}

\subjclass{Primary 11T23; Secondary 20G05}

\begin{abstract} 
We investigate the distribution of the angles of Gauss sums attached to the cuspidal representations of general linear groups over finite fields. In particular we show that they happen to be equidistributed w.r.t.the Haar measure. However, for representations of $PGL_2(\bb F_q)$, they are clustered around $1$ and $-1$ for odd $p$ and around $1$ for $p=2$.
\end{abstract}

\maketitle
Let $p$ be a prime number, $q$ a power of $p$ and $\F_q$ a finite field with $q$ elements. Fix a non-trivial additive character $\psi_p: \F_p\to \C^*$, and let $\psi_q$  the character on $\F_q$ given by composition with trace. Given a non-trivial  multiplicative character $\chi: \F_q^*\to \C^*$, the Gauss sum, 
\[g(\chi, \psi) =\sum_{a\in \F_q^*} \chi(a)\psi_q(a),\]
has absolute value $\sqrt{q}$. The following equidistribution theorem for the angles of Gauss sums (\cite[1.3.3]{Katz_sommes_exponentielles}) is a consequence  of Deligne's bound for Kloosterman sums obtained from his work on Weil conjectures: 

\begin{theorem}\label{theorem:Deligne}
As $q$ tends to infinity, the set of $(q-2)$ points $\{g(\chi, \psi_q)/\sqrt{q}: \chi ~\mbox{non-trivial}\}$ is equidistributed with respect to the normalized Haar measure $\frac{1}{2 \pi}dx$ on the unit circle $S^1$, i.e., for any continuous function $f:S^1 \to \C^*$, 
\[ \frac{1}{2\pi} \int_{S^1} f(x)dx=\lim_{q\to \infty} \frac{1}{q-2}\sum_{\chi\neq 1} f\left( g(\chi, \psi_q)/\sqrt{q}\right).\]
\end{theorem}

In this article we will study the equidistribution properties of the angles of Gauss sums attached to irreducible cuspidal representations of general linear groups over finite fields. For a natural number $d$, denote by $\psi_d$ (or just $\psi$ by abuse of notation), the additive character of the ring of $d \times d$ matrices $M_d(\F_q)$ defined by $\psi_d=\psi\circ \mbox{Tr}$. Given a complex representation $\rho$ of $GL_d(\F_q)$ of degree $N$, the matrix valued Gauss sum was introduced by Lamprecht (\cite{Lamprecht}): 
\begin{equation} \label{GSmatrix}
 G(\rho, \psi):= \sum_{x \in GL_d(\F_q)} \rho(x)\psi(\mbox{Tr}(x)) \in M_N(\C). 
\end{equation}
Suppose $\rho$ is irreducible. Then by Schur's lemma $G(\rho, \psi)$ is a scalar matrix:
\begin{equation}\label{GS}
 G(\rho, \psi)= g(\rho, \psi)I_N,
\end{equation}
where $g(\rho, \psi) \in \C$ and $I_N$ is the $N\times N$ identity matrix.

The irreducible characters of $GL_d(\F_q)$ were classified by Green \cite{Green_s_paper} in terms of conjugacy classes. Using Green's work, Kondo showed
\[
    |g(\rho,\psi)|=q^{\tfrac{d^2-\kappa(\rho)}{2}},
\]
where $\kappa(\rho)$ is the multiplicity of $1$ as a root of the characteristic polynomial attached to the conjugacy class of $GL_d(\mathbb F_q)$ corresponding to the irreducible representation $\rho$. Let $a(\rho):= q^{-\tfrac{d^2-\kappa(\rho)}{2}}g(\rho,\psi)$ denote the `angle', i.e,  the normalized Gauss sum of absolute value $1$ attached to $\rho$. 

We recall that an irreducible representation $(\rho, V)$ of $GL_d(\F_q)$ is said to be cuspidal, if the space of $N(\F_q)$-invariants $V^{N(\F_q)}$ is zero,  where $N$ is the unipotent radical of a parabolic subgroup $P$ of  $GL_d$ defined over $\F_q$. The cuspidal representations are the building blocks for all representations of $GL_d(\F_q)$, in the sense that any irreducible representation of $GL_d(\F_q)$ occurs in a representation parabolically induced from a cuspidal representation of the Levi component of a suitable parabolic. Via the Green correspondence, cuspidal representations correspond to elliptic conjugacy classes in $GL_d(\F_q)$, i.e., conjugacy classes of semisimple elements of 
$GL_d(\F_q)$ whose characteristic polynomial is irreducible over $\F_q$. For such representations, $\kappa(\rho)=0$.  We show an analogue of Deligne's theorem for the angles of the Gauss sums corresponding to cuspidal representations of $GL_d(\F_q)$:
\begin{theorem}\label{theorem:main}
Let $R_0(d, q)$ denote the set of isomorphism classes of irreducible, cuspidal representations of $GL_d(\F_q)$. The set of normalized Gauss sums $\{q^{-d^2/2}g(\rho, \psi)\mid \rho\in R_0(d,q)\}$ is equidistributed with respect to the normalized Haar measure on $S^1$, as $q$ tends to infinity. 
\end{theorem}

We now consider equidistribution results for irreducible representations of $GL_d(\F_q)$ with trivial central character. Such representations can be considered as representations of $PGL_d(\F_q)$. 
\begin{theorem}\label{theorem:centralchar}
Let $R_0^0(d, q)$ denote the set of isomorphism classes of irreducible, cuspidal representations of $GL_d(\F_q)$ with trivial central character. As $q$ tends to infinity, the set of normalized Gauss sums $\{q^{-d^2/2}g(\rho, \psi)\mid \rho\in R_0^0(d,q)\}$ is equidistributed with respect to the normalized Haar measure on $S^1$ for $d\geq 3$.

When $d=2$ and $p$ is odd, these normalized Gauss sums are equidistributed,  as $q$ tends to infinity, with respect to the measure 
\[\frac{1}{2}(\delta_1+\delta_{-1}),\] 
where $\delta_a$ denotes the Dirac measure supported at $a\in S^1$.  

When $d=2, ~p=2$, these normalized Gauss sums are equidistributed,  as $q$ tends to infinity with respect to the Dirac measure $\delta_1$ supported at $1$. \end{theorem}
The answer for $d=2$ and $p$ odd, suggests that the family of cuspidal representations of $PGL(2,\F_q)$ should split into two families, such that the normalized Gauss sums belonging to these subsets should be equidistributed with respect to $\delta_1$ and $\delta_{-1}$ respectively. The cuspidal representations of  $GL(2,\F_q)$ with trivial central character are parametrized by pairs of characters 
\[\{\chi,\chi^{\sigma}\} \quad\mbox{of $\F_{q^2}^*$, such that} \quad \chi\neq \chi^{\sigma} 
\quad \mbox{and}  \quad \chi|_{\F_q^*}=1,\]
where $\sigma$ is the non-trivial element of $\mbox{Gal}(\F_{q^2}/\F_q)$. Denote by $\rho_{\chi}$ the representation corresponding to $\chi$ as above. The collection of characters $\chi$ of  $\F_{q^2}^*$ whose restriction to  $\F_{q}^*$ is trivial is a cyclic group of order $q+1$. Of this, only the trivial and the quadratic characters are equivalent to their own Galois conjugates. The natural guess regarding the above expectation turns out to be valid:
\begin{theorem} \label{theorem:rep-square}
Let $p$ be an odd prime. Let $C_{pr}^0(2,q)$ denote the collection of primitive characters of $\F_{q^2}^*$ with respect to $\F_q$, whose restriction to $\F_q^*$ is trivial. As $q$ tends to infinity, the set of normalized Gauss sums $\{g(\rho_{\chi}, \psi)/q^2\}$ where $\chi\in C_{pr}^0(2,q)$ is a square of a character of  $\F_{q^2}^*$ (resp. non-square )  is equidistributed with respect to the Dirac measure $\delta_1$ (resp. $\delta_{-1}$) supported at $1$ (resp. $-1$) of $S^1$.  
\end{theorem}

\begin{remark}
One reason to study Gauss sums is their relation to $L$-functions. Specifically, Gauss sums $g(\chi,\psi)$ have absolute value $q^\frac{1}{2}$, and they satisfy the Hasse-Davenport relation: $-g(\chi \circ N_{\bb F_{q^n}^*/ \bb F^*_q}, \psi\circ Tr_{\bb F_{q^n}/ \bb F_q})=(-g(\chi,\psi))^n$. The first one says that the $L$-function of $\bb F_q[T]$ satisfies the Riemann hypothesis, while the second relation follows from the Euler product expansion (\cite{Ireland_Rosen}).

The Sato-Tate type of conjectures in automorphic forms predict for a tempered cusp form $\pi$ on a connected, reductive algebraic group $G$ defined over a number field $K$, that the Langlands-Satake parameters attached to an unramified component $\pi_v$ at finite places $v$ of $K$, should be equidistributed with respect to the projection of the Haar measure of $M$ onto its set of conjugacy classes, where $M$ is a maximal compact subgroup of the connected component of identity of the Langlands dual $^LG$ over $\C$. This led us to consider  the equidistribution question for cuspidal representations with trivial central character. Such representations can be considered as representations of the adjoint group $PGL_d(\F_q)=GL_d(\F_q)/\F_q^*$. The identity component of the Langlands dual group of $PGL_d$ is $SL_d$, whereas that of $GL_d$ is $GL_d$ itself. 

However, this reasoning does not seem to suggest the above answer.  We refer also to Remark \ref{remark_epsilon_factor} for a connection with epsilon factors. 
\end{remark}

\section{Preliminaries and reduction to abelian Gauss sums} In this section we recall relevant results about equidistribution, representation theory of $GL_d(\F_q)$, Gauss sums, and reduce the statement of the theorems to one involving the classical abelian Gauss sums. 
\subsection{Equidistribution} Let $X$ be a compact, Hausdorff  space, and $\mu$ a normalized Borel measure on $X$. A sequence $S_N$ of subsets of $X$ is said to be {\em equidistibuted} in $X$ with respect to $\mu$ if, for any continuous function $f$ on $X$, 
\begin{equation}
    \lim_{N \longrightarrow \infty} \frac{1}{\abs{S_N}} \sum\limits_{x \in S_N} f(x)=\int_{X}f d\mu.
\end{equation}
Suppose $\{f_i\}_{i\in I}$ is a family of continuous functions whose linear spans are dense in $C(X)$, the space of continous functions on $X$ with respect to the supremum norm. Assume further that the limits, 
\[\lim_{N \longrightarrow \infty} \frac{1}{\abs{S_N}} \sum\limits_{x \in S_N} f_i(x),\]
exists for all $i\in I$. Then there exists an unique measure $\mu$ on $X$ such that the sequence $S_N$ is equidistributed with respect to $\mu$ (\cite[Appendix, Chapter 1]{Serre_book}). Taking $X=S^1$, it suffices to work with the functions $z\mapsto z^n$ for $n\in \Z$. Hence, to check that the sequence of subsets $S_N$ of $S^1$ is equidistributed with respect to the normalized Haar measure on $S^1$, it suffices to show the following: 
\begin{equation} \label{equation:eqdisHaar}
   \lim_{N \longrightarrow \infty} \frac{1}{\abs{S_N}}\sum\limits_{z \in S_N} z^n = 0\quad  \mbox{for}~ n \neq 0, ~n\in \Z.
\end{equation}
The sequence $S_N$ is equidistributed with respect to the Dirac measure supported at $1$ (resp. $-1$) provided 
\begin{equation}
   \lim_{N \longrightarrow \infty} \frac{1}{\abs{S_N}}\sum\limits_{z \in S_N} z^n = 1\quad \mbox{(resp. $(-1)^n$)} \quad  \mbox{for $n\in \Z$}.
\end{equation}

If we further know that the sets $S_N$ are closed under complex conjugation, we need to verify the foregoing equations only for $n\geq 0$. 

\subsection{Kloosterman sums and Deligne's bound.}
Let $\psi$ be a nontrivial additive character of $\Fq$ into $\C ^*$. For $a \in \Fq$, the \textit{Kloosterman sum} is defined to be the sum
\begin{equation}
    \kln(a,q):=\smlm{\substack{x_i \in \Fq\\ x_1 \cdots x_n=a}} \psi(x_1+\ldots + x_n)
\end{equation}
Kloosterman sums are related to Gauss sums in the following way:
\begin{equation}\label{equation:GaussandKlooseterman}
    g(\chi,\psi)^n=\smlm{a \in {\Fq}^*} \chi(a) \kln(a,q).
\end{equation}
That is, powers of Gauss sums are the Fourier transform (over the group $\bb F_q^*$) of Kloosterman sums. The absolute values of Gauss sums is known classically: $\abs{g(\chi,\psi)}=\sqrt{q}$ if $\chi \neq \mathbf 1$ and $g(\mathbf {1}, \psi)=-1$. Hence by Parseval identity, we have the equality
\[
    \smlm{a} \abs{\kln(a,q)}^2=q^n \pm \frac{1}{q-1}.
\]
From this we can conclude crudely that $\abs{\kln(a)} =O(q^{n/2})$, and that $\abs{\kln(a)}^2 \leq \tfrac{q^n}{q-1} \pm \tfrac{1}{(q-1)^2}$ for at least one $a$. 

Deligne has obtained a uniform bound of $O(q^{\tfrac{n-1}{2}})$ for all $a$ as a consequence of his work on Weil conjectures: 
 
\begin{theorem}\label{Deligne_bound}\cite[Equation 7.1.3]{SGA_4_half}
If $a \in \bb F_q^*$, then
\begin{equation}
    \abs{\kln(a,q)} \leq nq^{\frac{n-1}{2}}.
\end{equation}
\end{theorem}
This saving of an order of $\tfrac{1}{2}$ is optimal and crucial in the computations that follow.

\subsection{Deligne-Lusztig representations} The classification of the  irreducible characters of $GL_d(\Fq)$ was carried out by Green (\cite{Green_s_paper}). More generally, the classification of irreducible representations of finite groups of Lie type was carried out by Deligne and Lusztig (\cite{Del_Lusztig_main_paper}) using geometric methods. In this article, we will use the Deligne-Luztig parametrization of irreducible representations. 

Choose a prime $\ell$ different from $p$. We work over the field $\Qlb$ instead of $\C$. 
The maximal tori of $GL_d$ over $\Fq$ up to $GL_d(\Fq)$ are classified in terms of conjugacy classes of the Weyl group $W\simeq S_d$ of $GL_d$. For $w\in S_d$, let ${\mathbb T}_w$ be the associated tori, where for the identity element in $S_d$, we associate the diagonal torus. Let $T_w= {\mathbb T}_w(\Fq)$. Associated to a character $\chi: T_w\to \Qlb^*$, Deligne and Lusztig construct a virtual representation $R^{\chi}_w$ (or $R^{\chi}_{T_{w}}$) of $GL_d(\Fq)$, and showed that every irreducible representation $\rho$ of $ GL_d(\Fq)$ is a constituent of some (not necessarily unique) $R^\chi_w$. 

\subsection{Cuspidal representations}
\begin{defn}
For a reductive group $\bb G$ defined over $\bb F_q$ with Frobenius $F$, we say that an irreducible representation $G=\bb G^F$ is \textbf{cuspidal}, if for every proper parabolic subgroup $\bb P$ of $\bb G$ with unipotent radical $\bb U$, the set of fixed vectors $\rho^{\bb U(\Fp)}$ is $0$.
\end{defn}
A theorem due to Harish Chandra says that cuspidal representations are the building blocks of all representations of $G$, in other words, the non cuspidal irreducible representations are induced from cuspidal representations with respect to suitable proper parabolic subgroups \cite{Bhama_book}.

Let $E$ be a finite extension of $\F_q$. A character $\chi:\F_q^*\to \C^*$ is said to be {\em primitive} if it does not factor via the norm map $E^*\to E'^*$ for any subextension $\F_q\subset E'\subset E$. Equivalently, for any non-trivial  $\sigma\in \mbox{Gal}(E/\F_q)$, the conjugate  $\chi^{\sigma}$ is not equal to the trivial character. Define two characters to be {\em equivalent} if they are equal up to a Galois twist. 

The Deligne-Lusztig parametrization of cuspidal representations of $GL_d(\F_q)$ is as follows: let $w_0\in S_d$ be a $d$-cycle, an element of longest length with respect to the standard generators for $S_d$.  The tori attached to $w_0$ corresponds to an embedding of the Weil restriction of scalars 
$R_{\F_{q^d}/\F_q}(\G_m)$ embedded inside $GL_d$. There is an isomorphism $T_{w_0}\simeq \F_{q^d}^*$. Such tori are not contained in any proper parabolic subgroup. 

Given a primitive character $\chi$ of $\F_{q^d}^*$, the Deligne-Lusztig representation $\varepsilon_{\bb G} \varepsilon_{\bb T_{w_0}} R^{\chi}_{w_0}$  is an irreducible, cuspidal representation, where $\varepsilon_{\bb G}:=(-1)^{\bb F_q-\textrm{rank of}\,\, \bb G}$. For $GL_d$, we have $\varepsilon_{GL_{d}}=(-1)^d$ and for the torus ${\bb T}_{w_0}$, the $\bb F_q$-rank is $-1$, hence $\varepsilon_{\bb T_{w_0}}=-1$, hence $\varepsilon_{\bb G} \varepsilon_{\bb T_{w_0}} R^{\chi}_{w_0}=(-1)^{d+1}R^{\chi}_{w_0}$. Further, if the characters $\chi$ and $\chi'$ are not conjugate under the action of ${\rm Gal}(\F_q^d/\F_q)$, the irreducible representations $(-1)^{d+1}R^{\chi}_{w_0}$ and $(-1)^{d+1}R^{\chi'}_{w_0}$ are distinct. This sets up a bijective correspondence between equivalence classes of primitive characters of $\F_{q^d}^*$ and isomorphism classes of cuspidal representations of $GL_d(\F_q)$. \cite[Chapter 6]{Bhama_book}

\subsection{Central character}
Let $\rho$ be an irreducible representation of a group $G$ whose centre is $Z$. By Schur's Lemma $\rho(z)$ is a scalar matrix for all $z \in Z$. That is, there exists a character $\chi: Z \to \C^*$ such that $\rho(z)=\chi(z)\cdot \textrm{Id}$. The character $\chi$ is called the central character of $\rho$.

Suppose $\bb T$ is a maximal $\F_q$-torus of $\bb G=GL_n$ (more generally of any reductive group $G$). Let $\bb B=\bb T \bb U$ be a Borel subgroup of $\bb G$. Denote by $\cL :\bb G \to \bb G$ the Lang isogeny $x\mapsto x^{-1}F(x)$, where $F$ is the Frobenius morphism. On the space $\cL^{-1}(U)$, the product $\bb G(\F_q)\times \bb T(\F_q)$ acts by $(g,t)(x)=gxt$. Let $Z$ be the center of $\bb G$. Restricted to $Z(\F_q)$, the induced action is same as the action of $Z(\F_q)\subset \bb T(\F_q)$. 

Thus, given a character $\chi$ of $T=\bb T(\F_q)$, the action of $Z(\F_q)$ on the $\chi$-isotypical component of $\ell$-adic  cohomology groups with compact support $H_c^i(\cL^{-1}(U), \bar{\Q}_{\ell})$ is given by the character $\chi$. 
Since the Deligne-Lusztig representations $R_T^{\chi}$ are afforded on the space 
\[H_c^*(\cL^{-1}(U), \bar{\Q}_{\ell})_{\chi}=\sum_{i=0}^{2 \textrm{dim}(U)} (-1)^iH_c^i(\cL^{-1}(U), \bar{\Q}_{\ell})_{\chi},\]
it follows that the center $Z(\F_q)$ acts via the character $\chi$ in the representation $R_T^{\chi}$.

\subsection{$GL_d$-Gauss sums} 
Suppose $\bb T$ is a maximal $\F_q$ torus in $GL_d$. Given a character $\chi: \bb T(\F_q)\to \bar{\Q}_{\ell}^*$, define the (abelian) Gauss sum, 
\begin{equation} 
 g(\chi, \psi):= \sum_{x \in \bb T(\F_q)} \chi(x)\psi(\mbox{Tr}(x)), 
\end{equation}
where $\mbox{Tr}$ denotes the trace of $x$ considered as a matrix element. 
It was shown by Kondo (\cite{Kondo_original_paper}) using Green's work and by 
Braverman and Kazhdan (\cite{BraKa}) using the theory of character sheaves of Lusztig, that the Gauss sums defined as in  Equations (\ref{GSmatrix}, \ref{GS}) for an irreducible representation $\rho$ of $GL_d(\Fq)$ is essentially the abelian Gauss sum. 
\begin{theorem}[Kondo, Braverman-Kazhdan]\label{theorem:KBK}\cite[Theorem 1.3]{BraKa}
Suppose $\rho$ is an irreducible representation of  $GL_d(\Fq)$ that is an irreducible constituent of $R_T^{\chi}$ for some maximal $\F_q$ torus $T$ in $GL_d$. Then, 
\[g(\rho, \psi)=q^{\tfrac{d^2-d}{2}}
g(\chi,\psi).\]
\end{theorem}

\subsection{Reduction to the abelian case} \label{section:reduction}
The cuspidal representatios $\rho$ are of the form $\pm R^{\chi}_{w_0}$ for some primitive character $\chi$ of $\F_{q^d}^*$. Via the correspondence between cuspidal representations  and primitive characters, the $d$ inequivalent Galois conjugates of $\chi$, give rise to isomorphic cuspidal representations. Further, by Theorem \ref{theorem:KBK},
\[g(\rho, \psi) =q^{(d^2-d)/2} g(\chi, \psi). \]
Thus, questions about equidistribution of the normalized Gauss sums  $\{g(\rho, \psi)/q^{d^2/2}\mid \rho\in R_0(d,q)\}$ are reduced to questions about  equidistribution of $\{g(\chi, \psi)/q^{d/2}\}$, where $\chi$ runs over the primitive characters of $\F_{q^d}^*$.  Theorems \ref{theorem:main}, \ref{theorem:centralchar} and \ref{theorem:rep-square} are consequences of the following theorems:

\begin{theorem}\label{theorem:main-charvers}
Let $C_{pr}(d,q)$ denote the collection of primitive characters of $\F_{q^d}^*$ with respect to $\F_q$. The set of normalized Gauss sums $\{g(\chi, \psi)q^{-d/2}\mid \rho\in C_{pr}(d,q)\}$ is equidistributed with respect to the normalized Haar measure on $S^1$, as $q$ tends to infinity. 
\end{theorem}

\begin{theorem}\label{theorem:centralchar-charvers}
Let $C_{pr}^0(d,q)$ denote the collection of primitive characters of $\F_{q^d}^*$ with respect to $\F_q$, whose restriction to $\F_q^*$ is trivial. Suppose $d\geq 3$. As $q$ tends to infinity, the set of normalized Gauss sums $\{g(\chi, \psi)q^{-d/2}\mid \rho\in C_{pr}^0(d,q)\}$ is equidistributed with respect to the normalized Haar measure on $S^1$. 
\end{theorem}

\begin{theorem} \label{theorem:evenodd}
Let $p$ be an odd prime. Let $C_{pr}^0(2,q)$ denote the collection of primitive characters of $\F_{q^2}^*$ with respect to $\F_q$, whose restriction to $\F_q^*$ is trivial. As $q$ tends to infinity, the set of normalized Gauss sums $\{g(\chi, \psi)/q\}$ where $\chi$ is a square of a character of  $\F_{q^2}^*$ (resp. non-square )  is equidistributed with respect to the Dirac measure $\delta_1$ (resp. $\delta_{-1}$) supported at $1$ (resp. $-1$) of $S^1$.  

When $p=2$, the set of normalized Gauss sums $\{g(\chi, \psi)/q\mid \chi \in C_{pr}^0(2,q) \}$ is equidistributed with respect to the Dirac measure $\delta_1$ supported at $1$ as $q\to \infty$. 
\end{theorem}

\section{Proof of Theorem \ref{theorem:main-charvers}}
In order to prove Theorem \ref{theorem:main}, we are reduced to proving Theorem \ref{theorem:main-charvers} showing the equidistribution of $\{g(\chi, \psi)/q^{d/2}\}$, where $\chi$ runs over the primitive characters of $\F_{q^d}^*$. 
We need to show for any non-zero integer $n$ that 
\begin{equation} \label{eqn:primeqdis}
\frac{1}{|C_{pr}(d,q)|} \smlm{\chi\,\, \textrm{primitive}} \Big( \frac{g(\chi, \psi)}{q^{d/2}} \Big)^n \longrightarrow 0,
\end{equation}
as $q \to \infty$. Since $g(\chi, \psi)^{-1}=g(\bar{\chi}, \bar{\psi})/q^d$, it suffices to verify the above limits for $n\geq 1$. 

Suppose $\chi$ is a non-primitive character of $\F_{q^d}^*$.  Then $\chi$ is left invariant by a non-trivial subgroup $H$ of $\mbox{Gal}(\F_{q^d}/\F_q)$. Hence $\chi$ is of the form $\chi'\circ N_{\F_{q^d}/\F_{q^d}^H}$, for some character $\chi'$ of the fixed field $\F_{q^d}^H$ under $H$. Here for an extension of finite fields $E/F$, $N_{E/F}$ denotes the norm map from $E^*$ to $F^*$. The cardinality of $\F_{q^d}^H$ is at most $q^{d/2}$. Since there are only finitely many fields contained in the finite extension $\F_{q^d}/\F_{q}$ independent of $q$,  the number of non-primitive characters is bounded by $Cq^{d/2}$ for some constant $C$. 

Thus the number of primitive characters is at least $q^d-Cq^{d/2}$. For any non-trivial character $\chi$ of $\F_{q^d}^*$, $|g(\chi, \psi)|=q^{d/2}$ and $g(1,\psi)=-1$.  Hence, in order to show the validity of Equation (\ref{eqn:primeqdis}), we can work with the set of all characters of $\F_{q^d}^*$. 
From Equation (\ref{equation:GaussandKlooseterman}) relating Gauss and Kloosterman sums, 
we are reduced to showing for $n\geq 1$, 
\[ \frac{1}{q^{d+nd/2}} \sum_{\chi\neq 1} ~~\sum_{a\in \F_{q^d}^*} \chi(a)Kl_n(a, q^d)\rightarrow 0, \]
as $q\to \infty$. Interchanging the order of summation and by orthogonality, 
\[ \frac{1}{q^{d+nd/2}} \sum_{\chi\neq 1} ~~\sum_{a\in \F_{q^d}^*} \chi(a)Kl_n(a, q^d)=\frac{(q^d-1)}{q^{d+nd/2}}Kl_n(1, q^d).\]
Theorem \ref{theorem:main} now follows by appealing to Deligne's bound given by Theorem \ref{Deligne_bound}, 
\[ |Kl_n(1, q^d)|\leq nq^{d(n-1)/2}.\]

\section{Proof of Theorem \ref{theorem:centralchar-charvers} for $d\geq 3$}
We now consider equidistribution of Gauss sums of cuspidal representations with trivial central character. Equivalently by the arguments in Section \ref{section:reduction}, we consider equidistribution of Gauss sums of primitive characters of of $\F_{q^d}^*$ that restrict trivially to $\F_q^*$. As seen in the foregoing section, the cardinality of non-primitive characters is of order $O(q^{d/2})$. The group $C^0(d,q)$ of characters of $\F_{q^d}^*$ that restrict trivially to $\F_q^*$ is of order $(q^d-1)/(q-1)$. 

Suppose $d\geq 3$. Then the ratio of the number of non-primitive characters to that of primitive characters with trivial central character goes to zero as $q$ tends to infinity. Arguing as in the proof of Theorem \ref{theorem:main-charvers} given above,  in order to prove Theorem \ref{theorem:centralchar-charvers}, it is sufficient to work with the entire character group $C^0(d,q)$ rather than restrict to just the primitive elements in $C^0(d,q)$. Thus we are reduced to showing 
for any natural number $n\geq 1$ that 
\begin{equation} \label{eqn:?}
\frac{1}{|C^{0}(d,q)|} \smlm{\chi\in C^0(d,q)} \Big( \frac{g(\chi, \psi)}{q^{d/2}} \Big)^n \longrightarrow 0,
\end{equation}
as $q \to \infty$. In terms of Kloosterman sums, we need to show 
for $n\geq 1$, that 
\begin{equation}\label{eqn:d3} 
\frac{1}{|C^0(d,q)|q^{nd/2}} \sum_{\chi\in C^0(d,q)} ~~\sum_{a\in \F_{q^d}^*} \chi(a)Kl_n(a, q^d)\rightarrow 0, 
\end{equation}
as $q\to \infty$.

Let $\widehat{\F_{q^d}^*}$ denote the group of characters of ${\F_{q^d}^*}$. Under the non-degenerate pairing, 
\begin{equation}\label{eqn:pairing}
 {\F_{q^d}^*}\times \widehat{\F_{q^d}^*}\to \C^*,
\end{equation}
the left annihilator of the group $C^0(d,q)$ is precisely $\F_q^*$. Hence, 
\[
    \sum_{\chi\in C^0(d,q)}\chi(a)=\begin{cases} |C^0(d,q)| & \text{if $a\in \F_q^*$}\\
0& \text{if $a\not\in \F_q^*$.}
\end{cases}
\]
Interchanging the order of summation in Equation (\ref{eqn:d3}), we get
\[ \frac{1}{|C^0(d,q)|q^{nd/2}} \sum_{\chi\in C^0(d,q)}~~ \sum_{a\in \F_{q^d}^*} \chi(a)Kl_n(a, q^d)=\frac{1}{q^{nd/2}}  \sum_{a\in \F_{q}^*} Kl_n(a, q^d).\]
Using Deligne's bound, 
\[\frac{1}{q^{nd/2}}  \sum_{a\in \F_{q}^*} |Kl_n(a, q^d)|\leq \frac{(q-1)nq^{d(n-1)/2}}{q^{nd/2}}=\frac{n(q-1)}{q^{d/2}}.\]
Since $d\geq 3$, the term goes to zero and this proves Theorem \ref{theorem:centralchar-charvers}. 

\section{Proof of Theorem \ref{theorem:evenodd}}
Let $C^0=C^0(2,q)$ be the subgroup of characters of $\F_{q^2}^*$  that is trivial restricted to $\F_q^*$. Its cardinality is $q+1$. Suppose $\chi$ is a character of $\F_q^*$ such that $\chi\circ N_{\F_q^2/\F_q}$ belongs to $C^0$. This is equivalent to saying that $\chi(x^2)$ is identically $1$. When $p=2$, the only  such character is the trivial character, and when $p$ is odd the condition implies that $\chi$ is quadratic. Hence the number of primitive characters of $\F_{q^2}^*$ that are trivial upon restriction to $\F_q^*$ is precisely $(q-1)$ when $p$ is odd and $q$ when $p$ is even. Since the set of non-primitive characters is at most $2$, arguing as before, in order to prove to equidistribution we can work with $C^0$ rather than just its subset of primitive characters. 

For odd $p$, let
\[\begin{split}
C^{0,s}&=\{\chi\in C^0\mid \chi=\eta^2 \quad\mbox{for some $\eta\in C^0$}\}\\
C^{0,ns}&=\{\chi\in C^0\mid \chi\neq \eta^2 \quad\mbox{for any $\eta\in C^0$}\}.
\end{split}
\]
Both sets have cardinality $(q+1)/2$ when $p$ is odd. In what follows, when $p=2$, by an abuse of notation, we will work with $C^{0,s}$-version, where we take $C^{0,s}$ to be equal to $C^0$. 
The quadratic character belongs to $C^{0,s}$ precisely when $4|(q+1)$. 

In order to prove Theorem \ref{theorem:evenodd} for $p$ odd,  we need to show that for $n\geq  1$, the average of the $n$-th moments of the Gauss sums, have the following limiting behaviour as $q\to \infty$:
\begin{align}
\frac{1}{|C^{0,s}|} \smlm{\chi\in C^{0,s}} \Big( \frac{g(\chi, \psi)}{q} \Big)^n&=\frac{2}{(q+1)q^{n}} \sum_{\chi\in C^{0,s}} ~~\sum_{a\in \F_{q^2}^*} \chi(a)Kl_n(a, q^2)\rightarrow 1, \label{eqn:cosmoment} \\
\text{and} \quad \frac{1}{|C^{0,ns}|} \smlm{\chi\in C^{0,ns}} \Big( \frac{g(\chi, \psi)}{q} \Big)^n&=\frac{2}{(q+1)q^{n}} \sum_{\chi\in C^{0,ns}} ~~\sum_{a\in \F_{q^2}^*} \chi(a)Kl_n(a, q^2)\rightarrow (-1)^n.\label{eqn:consmoment}
\end{align} 
For $p=2$, we need to show that as $q\to \infty$, 
\[ \frac{1}{q^{n+1}} \sum_{\chi\in C^{0}} ~~\sum_{a\in \F_{q^2}^*} \chi(a)Kl_n(a, q^2)\rightarrow 1. \]

The left annihilator of the subgroup $C^0$ of $\widehat{\F_{q^2}^*}$ with respect to the non-degenerate bilinear pairiing given by Equation (\ref{eqn:pairing}) is precisely $\F_q^*$. 
Let $\sigma$ be the non-trivial automorphism of $\F_{q^2}$ over $\F_q$. Define
\begin{equation*}\label{defn:S}
    \begin{split}
    S^0&=\{x\in \F_{q^2}^*\mid \sigma(x)=x\}=\F_q^*,\\
    S^{-}&=\{x\in \F_{q^2}^*\mid \sigma(x)=-x\}=\{x\in \F_{q^2}^*\mid Tr(x)=0\},\\
    S&=S^0\cup S^{-}=\{x\in \F_{q^2}^*\mid  x^2\in \F_q^*\}.
\end{split}
\end{equation*}
where $Tr$ denotes the trace from $ \F_{q^2}$ to $\F_q$. When $p=2$, $S^0=S^-$. The annihilator of the subgroup $C^{0,s}$ contains the set $S$. 
 When $p$ is odd, the cardinality of $S$ is $2(q-1)$, and hence it is precisely the annihilator of $C^{0,s}$. Consequently, for $x\in \F_{q^2}^*$, 
\begin{equation} \label{eqn:cosorth}
\sum_{\chi\in C^{0,s}}\chi(x)=\begin{cases} 
0 &  \text{if $x\not \in S$},\\
(q-1)/2 & \text{ if $x\in S$}.
\end{cases}
\end{equation}
The value of a character $\chi\in C^{0,ns}$ evaluated on an element $x\in S$ is given as, 
\begin{equation}\label{eqn:consvalue}
 \chi(x)=\begin{cases} 
1 &  \text{if $x\in S^0$},\\
-1& \text{ if $x\in S^-$}.
\end{cases}
\end{equation}
Since $C^{0,ns}$ is a coset of $C^{0,s}$ in $C^0$, equations (\ref{eqn:cosorth}) and (\ref{eqn:consvalue}) yield, 
\begin{equation} \label{eqn:consorth}
\sum_{\chi\in C^{0,ns}}\chi(x)=\begin{cases} 0 &  \text{if $x\not \in S$},\\
(q-1)/2 & \text{ if $x\in S^0$},\\
-(q-1)/2 & \text{ if $x\in S^-$}.
\end{cases}
\end{equation}
Define for $n\geq 1$, the following averages of Kloosterman sums: 
\[\begin{split}
\text{(Invariant)}\quad I_n&=\sum_{a\in \F_q^*}Kl_n(a, q^2)=
\sum_{\substack{x_1, \cdots, x_n \in \F_{q^2}^*\\
\prod_i x_i\in \F_q^*}} \psi\circ Tr(x_1+\cdots+x_n)\\
\text{(Anti-invariant)}\quad A_n&=
\sum_{\substack{a\in \F_{q^2}^*\\ Tr(a)=0}}Kl_n(a, q^2)=\sum_{a\in S^-}Kl_n(a, q^2)=\sum_{\substack{x_1, \cdots, x_n \in \F_{q^2}^*\\
Tr(\prod_i x_i)=0}}\psi\circ Tr(x_1+\cdots+x_n).
 \end{split}
 \]
Substituting the values obtained from Equations (\ref{eqn:cosorth},  \ref{eqn:consorth}) into Equations (\ref{eqn:cosmoment},  \ref{eqn:consmoment}) we need to show for $n\geq 1$ and $q\to \infty$, 
\begin{align}
\frac{1}{|C^{0,s}|} \smlm{\chi\in C^{0,s}} \Big( \frac{g(\chi, \psi)}{q} \Big)^n&=\frac{(q-1)}{(q+1)q^{n}} (I_n+A_n)\rightarrow 1, \label{eqn:cosmoment-IA} \\
\text{and} \quad \frac{1}{|C^{0,ns}|} \smlm{\chi\in C^{0,ns}} \Big( \frac{g(\chi, \psi)}{q} \Big)^n&=\frac{(q-1)}{(q+1)q^{n}} (I_n-A_n)\rightarrow (-1)^n.\label{eqn:consmoment-IA}
\end{align} 
We observe that \begin{equation}\label{eqn:I1odd}
I_1=-1\quad \text{and} \quad A_1=q-1. 
\end{equation}
Hence, 
\[ I_1+A_1= q-2\quad \text{and}\quad I_1-A_1= -q.\]
This gives us the validity of Equations (\ref{eqn:cosmoment-IA}) and (\ref{eqn:consmoment-IA}) for $n=1$. 

When $p=2$, 
\begin{equation} \label{eqn:co-orth2}
\sum_{\chi\in C^{0}}\chi(x)=\begin{cases} 0 &  \text{if $x\not \in \F_q^*$},\\
(q-1)& \text{ if $x\in \F_q^*$}.\end{cases}
\end{equation}
Thus we need to show that as $q\to \infty$, 
the sum
\begin{equation}\label{eqn:2moment-IA} 
\frac{1}{|C^{0}|} \smlm{\chi\in C^{0}} \Big( \frac{g(\chi, \psi)}{q} \Big)^n=\frac{(q-1)}{q^{n+1}} I_n\rightarrow 1.
\end{equation}
Further, 
\begin{equation}\label{eqn:I1even}
 I_1=\sum_{a\in \F_q^*}\psi\circ Tr(a)=\sum_{a\in \F_q^*}\psi(2a)=(q-1), 
\end{equation}
and this yields the validity of equation (\ref{eqn:2moment-IA}) when $n=1$.

The following key proposition gives a recurrence relation between the sums $I_n$ and $A_n$, allowing an inductive procedure to calculate their values explicitly: 

\begin{proposition}
For $n\geq 2$ and $p$ odd, the following recurrence relation holds:
\[ \begin{split}
I_n& =qA_{n-1}+(-1)^n\\
A_n& =qI_{n-1}+(-1)^n.
\end{split}
\]
When $p=2$, 
\[ I_n =qI_{n-1}+(-1)^n.\]
\end{proposition}
\begin{proof}[Proof of Theorem \ref{theorem:evenodd}]
Inductively, it follows for $n\geq 2$, 
\begin{align*}
I_n+A_n&= q(I_{n-1}+A_{n-1})+2(-1)^n \\
&=q^{n-1}(I_1+A_1)+2(-1)^n (1-q+q^2+\cdots +(-1)^{n-2}q^{n-2})\\
&=(q-2)q^{n-1}+2(-1)^n \frac{(1+(-q)^{n-1})}{(1+q)}.
\end{align*}
Similarly, 
\[I_n-A_n= -q(I_{n-1}-A_{n-1})=(-q)^{n-1}(I_1-A_1)=(-q)^n.\]
For $p=2$ and $n\geq 2$, we have 
\begin{align*}
I_n& = qI_{n-1}+(-1)^n\\
&=q^{n-1}I_1 +(-1)^n(1-q+q^2+\cdots +(-1)^{n-2}q^{n-2})\\
&=(q-1)q^{n-1}+(-1)^n\frac{(1+(-q)^{n-1})}{(1+q)}.
\end{align*}
Together with the values  for $n=1$, it  gives us the validity of Equations (\ref{eqn:cosmoment-IA}, \ref{eqn:consmoment-IA}, \ref{eqn:2moment-IA}) for $n\geq 1$, and with it a proof of Theorem \ref{theorem:evenodd} for all $p$. 
\end{proof}
\begin{remark}
Substituting the above expressions for $I_n\pm A_n$ in Equations (\ref{eqn:cosmoment-IA}) and (\ref{eqn:consmoment-IA}), we see that the $n$-th moments are not precisely equal to the limiting Dirac measure. Hence the normalized Gauss sums, even accounting for the contribution from the trivial character (equal to $-1/q^n$) are not identically equal to $1$ or $-1$, but only tend to the respective Dirac measures as $q \to \infty$.
\end{remark}

\begin{proof}[Proof of Proposition]
Suppose $\prod_{i=1}^nx_i=a$, where by $x_i$ we consider an arbitrary element in $\F_{q^2}^*$. Let $\prod_{i=1}^{n-1}x_i=b$ and $x_n=a/b$. With this decomposition, 
\begin{align*}
I_n&= \sum_{\substack{x_1, \cdots, x_n \in \F_{q^2}^*\\
\prod_i x_i\in \F_q^*}} \psi\circ Tr(x_1+\cdots+x_n)\\
&=\sum_{a\in \F_q^*}~~\sum_{\substack{x_1\cdots x_{n-1}\in \F_{q^2}^*\\
x_1\cdots x_{n-1}=b}} \psi\circ Tr(x_1+\cdots+x_{n-1})\psi\circ Tr(a/b)\\
&=\sum_{a\in \F_q^*}~~\sum_{b\in \F_{q^2}^*} Kl_{n-1}(b, q^2)\psi\circ Tr(a/b)\\
&=\sum_{b\in \F_{q^2}^*} ~~\sum_{a\in \F_q^*}Kl_{n-1}(b, q^2)\psi\circ Tr(a/b).
\end{align*}
Since $a\in \F_q^*$,
\[\psi\circ Tr(a/b)=\psi(a/b+\sigma(a/b))=\psi(aTr(b^{-1})).\]
Suppose $Tr(b^{-1})\neq 0$. As $a$ varies over $\F_q^*$, the collection of elements of the form $aTr(b^{-1})$ forms the entire set $\F_q^*$. Hence, 
\[ \sum_{a\in \F_q^*}\psi\circ Tr(a/b)=\begin{cases} -1 & \text{if $Tr(b)\neq 0$},\\
(q-1) &\text{if $Tr(b)=0$,}
\end{cases}
\]
where we have used the fact that $Tr(b^{-1})=0$ iff $Tr(b)=0$. Substituting this in the above expression for $I_n$, we get
\begin{align*}
I_n&= \sum_{b\in \F_{q^2}^*} ~~\sum_{a\in \F_q^*}Kl_{n-1}(b, q^2)\psi\circ Tr(a/b)\\
&=(q-1)\sum_{\substack{b\in \F_{q^2}^*\\Tr(b)=0}} Kl_{n-1}(b, q^2)-\sum_{\substack{b\in \F_{q^2}^*\\Tr(b)\neq 0}}Kl_{n-1}(b, q^2)\\
&=qA_{n-1}-\sum_{b\in \F_{q^2}^*}Kl_{n-1}(b, q^2)\\
&=qA_{n-1}-\sum_{x_1, \cdots, x_{n-1}\in \F_{q^2}^*}\psi\circ Tr(x_1+\cdots+x_{n-1})\\
&=qA_{n-1}-\prod_{i=1}^{n-1}\left(\sum_{x_i\in \F_{q^2}^*}\psi\circ Tr(x_i)\right)\\
&=qA_{n-1}- (-1)^{n-1}.
\end{align*}
We observe here that the above proof holds without any change for $p=2$, where we replace $A_{n-1}$ by $I_{n-1}$. 

The proof for $A_n$ is similar. The set of elements of trace $0$ in $\F_{q^2}$ consists of elements of the form $a\sqrt{\delta}$, where $\delta$ is a non-square in $\F_q^*$ and $a\in \F_q$. 
\begin{align*}
A_n&= \sum_{\substack{x_1, \cdots, x_n \in \F_{q^2}^*\\
Tr(\prod_i x_i)=0}}\psi\circ Tr(x_1+\cdots+x_n)\\
&=\sum_{a\in \F_q^*}~~\sum_{\substack{x_1\cdots x_{n-1}\in \F_{q^2}^*\\
x_1\cdots x_{n-1}=b}} \psi\circ Tr(x_1+\cdots+x_{n-1})\psi\circ Tr(a\sqrt{\delta}/b)\\
&=(q-1)\sum_{\substack{b\in \F_{q^2}^*\\Tr(\sqrt{\delta}/b)=0}} Kl_{n-1}(b, q^2)-\sum_{\substack{b\in \F_{q^2}^*\\Tr(\sqrt{\delta}/b)\neq 0}}Kl_{n-1}(b, q^2).
\end{align*}
Now 
\[Tr(\sqrt{\delta}/b)=\sqrt{\delta}/b+\sigma(\sqrt{\delta}/b)=\sqrt{\delta}/b-\sqrt{\delta}/\sigma(b)=0\]
precisely when $b=\sigma(b)$, i.e., $b\in \F_q^*$. Hence, 
\begin{align*}
A_n&=(q-1)\sum_{b\in \F_q^*} Kl_{n-1}(b, q^2)-\sum_{b\in \F_q^*}Kl_{n-1}(b, q^2)\\
&=qI_{n-1}- (-1)^{n-1}.
\end{align*}
\end{proof}

\begin{remark}\label{remark_epsilon_factor}
We now give a connection of the hypothesis of Theorem \ref{theorem:rep-square} to $\epsilon$ factors. We thank U. K. Anandavardhanan and Dipendra Prasad for pointing out this connection. We refer to the papers of Fröhlich and Queyrut and  Deligne (\cite{Fr-Qu}, \cite{Deligne_local_const_orthogonal} ) for further details. 
Let $K$ be a local field with residue field $k\cong \mathbb F_q$. Let $L$ be the unramified quadratic extension of $K$. There is a natural projection map $L^* \xrightarrow{\cong} \mathcal{O}_L^* \times \mathbb Z \to \mathbb F_{q^2}^*$. Via this projection, $\chi$ can be considered as a character of $L^*$, which we denote by $\chi'$. Assume now that $\chi$ restricts trivially to $\mathbb{F}_q^*$. Then $\chi'$ restricts trivially to $K^* \subset L^*$.

For a local field $K$, let $W_K$ denote its Weil group. By the isomorphism $W_L^{ab} \xrightarrow{\cong} L^*$, $\chi'$ can be considered as a character of $W_L^{ab}$, and therefore of $W_L$. Since $\chi'$ restricts trivially to $K^*$, the induced representation $\textrm{Ind}^{W_K}_{W_L}(\chi')$ can be realised over $\mathbb R$. Let $V=\textrm{Ind}^{W_K}_{W_L}([\chi']-[1])$ be the induction to $W_K$ of the virtual representation $[\chi']-[1]$. Then $V$ has dimension $0$ and determinant $1$. By Fröhlich-Queyrut and Deligne, the espsilon factor of $V$ satisfies
\begin{equation*}
    \epsilon(V,\tfrac{1}{2})=\chi'(\Delta),
\end{equation*}
for any element $\Delta$ of $L^*$ whose trace to $K$ vanishes. By taking $\Delta \in L^*$ to correspond to an element $\sqrt{\delta}$ of $S^- \subset \mathbb F_{q^2}^*$ defined as in the proof of Theorem \ref{theorem:evenodd}, we see that for $\chi \in C^0_{pr}$, $\chi$ is a square if and only if $\chi(\sqrt{\delta})=1$. This holds precisely when  $\chi'(\Delta)=1$,  or equivalently when  $\epsilon(V,\tfrac{1}{2})=1$.
\end{remark}

\bibliographystyle{alpha}

\begin{thebibliography}{sam}

\bibitem[BK03]{BraKa}
Alexander Braverman and David Kazhdan.
\newblock {$\gamma$}-sheaves on reductive groups.
\newblock {\em Studies in memory of {I}ssai {S}chur ({C}hevaleret/{R}ehovot,
  2000)}, 210:27--47, 2003.

\bibitem[Del76]{Deligne_local_const_orthogonal}
Pierre Deligne.
\newblock Les constantes locales de l'\'{e}quation fonctionnelle de la fonction
  {$L$} d'{A}rtin d'une repr\'{e}sentation orthogonale.
\newblock {\em Invent. Math.}, 35:299--316, 1976.

\bibitem[Del77]{SGA_4_half}
P.~Deligne.
\newblock Applications de la formule des traces aux sommes
  trigonom\'{e}triques.
\newblock In {\em Cohomologie \'{e}tale}, volume 569 of {\em Lecture Notes in
  Math.}, pages 168--232. Springer, Berlin, 1977.

\bibitem[DL76]{Del_Lusztig_main_paper}
P.~Deligne and G.~Lusztig.
\newblock Representations of reductive groups over finite fields.
\newblock {\em Ann. of Math. (2)}, 103(1):103--161, 1976.

\bibitem[FQ73]{Fr-Qu}
A.~Fr\"{o}hlich and J.~Queyrut.
\newblock On the functional equation of the {A}rtin {$L$}-function for
  characters of real representations.
\newblock {\em Invent. Math.}, 20:125--138, 1973.

\bibitem[Fre19]{Fresan}
Javier Fres\'{a}n.
\newblock \'{E}quir\'{e}partition de {S}ommes exponentielles [travaux de
  {K}atz].
\newblock {\em Ast\'{e}risque}, (414, S\'{e}minaire Bourbaki. Vol. 2017/2018.
  Expos\'{e}s 1136--1150):Exp. No. 1141, 205--250, 2019.

\bibitem[Gre55]{Green_s_paper}
J.~A. Green.
\newblock The characters of the finite general linear groups.
\newblock {\em Trans. Amer. Math. Soc.}, 80:402--447, 1955.

\bibitem[IR90]{Ireland_Rosen}
Kenneth Ireland and Michael Rosen.
\newblock {\em A classical introduction to modern number theory}, volume~84 of
  {\em Graduate Texts in Mathematics}.
\newblock Springer-Verlag, New York, second edition, 1990.

\bibitem[Kat80]{Katz_sommes_exponentielles}
Nicholas~M. Katz.
\newblock {\em Sommes exponentielles}, volume~79 of {\em Ast\'{e}risque}.
\newblock Soci\'{e}t\'{e} Math\'{e}matique de France, Paris, 1980.
\newblock Course taught at the University of Paris, Orsay, Fall 1979, With a
  preface by Luc Illusie, Notes written by G\'{e}rard Laumon, With an English
  summary.

\bibitem[Kon63]{Kondo_original_paper}
Takeshi Kondo.
\newblock On {G}aussian sums attached to the general linear groups over finite
  fields.
\newblock {\em J. Math. Soc. Japan}, 15:244--255, 1963.

\bibitem[Lam57]{Lamprecht}
Erich Lamprecht.
\newblock Struktur und {R}elationen allgemeiner {G}ausscher {S}ummen in
  endlichen {R}ingen. {I}, {II}.
\newblock {\em J. Reine Angew. Math.}, 197:1--26, 27--48, 1957.

\bibitem[Ser68]{Serre_book}
Jean-Pierre Serre.
\newblock {\em Abelian {$l$}-adic representations and elliptic curves}.
\newblock W. A. Benjamin, Inc., New York-Amsterdam, 1968.
\newblock McGill University lecture notes written with the collaboration of
  Willem Kuyk and John Labute.

\bibitem[Sri79]{Bhama_book}
Bhama Srinivasan.
\newblock {\em Representations of finite {C}hevalley groups}, volume 764 of
  {\em Lecture Notes in Mathematics}.
\newblock Springer-Verlag, Berlin-New York, 1979.
\newblock A survey.

\bibitem[Tat79]{Tate_article}
J.~Tate.
\newblock Number theoretic background.
\newblock In {\em Automorphic forms, representations and {$L$}-functions
  ({P}roc. {S}ympos. {P}ure {M}ath., {O}regon {S}tate {U}niv., {C}orvallis,
  {O}re., 1977), {P}art 2}, Proc. Sympos. Pure Math., XXXIII, pages 3--26.
  Amer. Math. Soc., Providence, R.I., 1979.

\end{thebibliography}

\nocite{*}
\end{document}